\documentclass[a4paper,12pt,reqno]{amsart}

\usepackage[english]{babel}
\usepackage[utf8]{inputenc}
\usepackage[headings]{fullpage}
\usepackage{multicol}
\usepackage[colorlinks,pdftex]{hyperref}

\DeclareMathOperator{\ad}{ad}
\DeclareMathOperator{\Ad}{Ad}
\DeclareMathOperator{\rank}{rank}
\DeclareMathOperator{\trace}{trace}
\newcommand{\so}{\mathfrak{so}}
\newcommand{\su}{\mathfrak{su}}

\theoremstyle{plain}
\newtheorem{theorem}{Theorem}
\numberwithin{theorem}{section}
\newtheorem{proposition}[theorem]{Proposition}
\newtheorem{lemma}[theorem]{Lemma}

\theoremstyle{definition}
\newtheorem{example}[theorem]{Example}

\theoremstyle{remark}
\newtheorem{remark}[theorem]{Remark}

\numberwithin{equation}{section}

\begin{document}
\title{Manifolds admitting a metric with co-index of symmetry $4$}

\author{Silvio Reggiani}
\address{CONICET and Universidad Nacional de Rosario, ECEN-FCEIA,
  Departamento de Ma\-te\-má\-ti\-ca. Av. Pellegrini 250, 2000
  Rosario, Argentina.}
\email{\href{mailto:reggiani@fceia.unr.edu.ar}{reggiani@fceia.unr.edu.ar}}
\urladdr{\url{http://www.fceia.unr.edu.ar/~reggiani}}

\date{\today}

\thanks{Supported by CONICET. Partially supported by ANPCyT and SeCyT-UNR}



\maketitle

\begin{abstract}
  By a recent result, it is known that compact homogeneous spaces with
  co-index of symmetry $4$ are quotients of a semisimple Lie group of
  dimension at most $10$. In this paper we determine exactly which
  ones of these spaces actually admit such a metric. For all the
  admissible spaces we construct explicit examples of these metrics.
\end{abstract}

\section{Introduction}

The problem of classifying the $G$-invariant Riemannian metrics on a
given homogeneous manifold $M = G/H$ is a difficult one. Even in the
case $M = G/\{e\}$ of a Lie group with a left invariant metric, this
problem is far from being solved. What makes more sense is to impose
some geometric constrains and restrict ourselves to a more manageable
class. For instance, we know exactly which Lie groups admit a
bi-invariant metric, and how a bi-invariant metric looks like in such
a group. More generally, if we ask for parallel tensor curvature, we
end up with Cartan's classification of the symmetric spaces
\cite{cartan-1926-1927}.

One possible way to approach this general problem, is by trying to
classify homogeneous spaces according to their index of symmetry,
first introduced in \cite{olmos-reggiani-tamaru-2014}. This proves to
be a fruitful way to address the issue, leading to very interesting
examples and strong structure results. Let us say quickly that the
index of symmetry is a geometric invariant, which measures how far is
a homogeneous Riemannian manifold from being a symmetric space. More 
precisely, the index of symmetry of a homogeneous Riemannian manifold
$M = G/H$ can be defined as the maximum number $i_{\mathfrak s}(M)$
of linearly independent Killing fields which are parallel at a given
point of $M$. Associated to this concept there is a $G$-invariant
distribution on $M$ called the distribution of symmetry, whose rank
equals $i_{\mathfrak s}(M)$, which is integrable with totally geodesic
leaves. Moreover, the leaves of the distribution of symmetry are
isometric to a globally symmetric space, called the leaf of symmetry
of $M$. The distribution of symmetry was computed for compact
naturally reductive spaces in \cite{olmos-reggiani-tamaru-2014} and
for naturally reductive nilpotent Lie groups in
\cite{reggiani18_distr_symmet_natur_reduc_nilpot_lie_group}. In
\cite{podesta-2015}, Podestá computed the index of symmetry for Kähler
metrics on generalized flag manifolds, showing that the leaf of
symmetry is a Hermitian symmetric space. There is also a
classification of left invariant metrics on $3$-dimensional unimodular
Lie groups according to their index of symmetry
\cite{Reggiani_2018}. Although there is some work in the non compact
setting, the most important structure results related to the index of
symmetry appear almost exclusively in the compact case (mainly
because of the existence of a bi-invariant metric on the full isometry
group). In particular, in the work \cite{berndt-olmos-reggiani-2017}
the classification of compact homogeneous spaces with co-index of
symmetry less or equal than $3$ is given (the co-index of symmetry of
$M$ is $\dim M - i_{\mathfrak s}(M)$). Namely, there are no spaces
with co-index $1$ (this is also the case for non compact spaces
according to \cite{Reggiani_2018}); all spaces with co-index of
symmetry $2$ are covered by $SU(2)$ with certain left invariant
metrics; and the spaces with co-index $3$ arise as certain
$SO(4)$-invariant metrics on $SO(4)/SO(2)$ (for the standard inclusion
of $SO(2)$ into $SO(4)$). In particular, in these cases the underlying
manifold supporting such metrics is the same, up to a cover.  These
results rely on a more general theorem proved in
\cite{berndt-olmos-reggiani-2017} which gives a bound on the dimension
of $M$ in terms  of its co-index of symmetry. More precisely, if $M$
is compact  homogeneous (without symmetric factors) of co-index of
symmetry $k$, then there exists a transitive semisimple Lie group $G'$
such that  
\begin{equation}\label{eq:3}
  \dim G' \le \frac{k(k + 1)}2.
\end{equation}

This is the reason why in the above cases there is only one possible
space admitting such metrics. The next logical step is to study spaces
with co-index of symmetry $4$. But in this case the situation is more
complicated, as there are several possibilities for the group
$G'$. The goal of this paper is to determine which homogeneous spaces
$G'/H'$, with $G'$ as in (\ref{eq:3}), admit a metric of co-index of
symmetry $k = 4$. By a simple inspection one can easily derive a list
of all the spaces $G'/H'$ which could admit a metric of co-index
$4$. Actually the list is somewhat shorter than one expects, as in
the extreme case where $\dim G' = 10$, the isotropy group must have
positive dimension. From this list we can exclude the spaces
$SO(5)/SO(2)$ and $SO(5)/(SO(3) \times SO(2))$. In order to do that,
we need to study the isotropy representation and the transvection
group of the possible leaf of symmetry (which have dimension $5$ and
$2$ respectively). For all the remaining cases we give explicit
metrics with co-index $4$. Some families of examples are constructed
from the classification given in \cite{berndt-olmos-reggiani-2017} for
co-index $3$ and the classification of naturally reductive spaces of
dimension $6$ \cite{agricola-ferreira-friedrich-2014}. Another
argument used in the construction of the metrics comes from the
so-called double symmetric pairs $G_1 \supset G_2 \supset G_3$, where
$G_1/G_2$ and $G_2/G_3$ are symmetric pairs. This trick is used in
\cite{olmos-reggiani-tamaru-2014}, where perturbing the normal
homogeneous metric on $G_1/G_3$, one sometimes gets a metric with leaf
of symmetry $G_2/G_3$. This argument does not always work, as one has
to prove every time that the proposed metric is not symmetric. Some
examples of this were known, but we can give a new one associated with
double symmetric pair $SO(5) \supset SO(4) \supset SO(2) \times
SO(2)$. Here the leaf of symmetry is a product of spheres.

Finally, we want to point out that the case $G' = SO(5)$ and $H' =
SO(3)$ provides examples of metric of co-index $4$, but only for the
standard inclusion of $SO(3)$ in $SO(5)$. In order to prove that, we
use an argument similar to that used in
\cite{berndt-olmos-reggiani-2017} for $SO(4)/SO(2)$, which involves
the so-called strongly symmetric autoparallel distributions (See
Theorem~\ref{sec:case-so5so3-1}).

\section{Preliminaries}

We use this section to fix some notation and review the structure
theory concerning the index of symmetry of a compact homogeneous
space. The main references for this section are
\cite{olmos-reggiani-tamaru-2014} and
\cite{berndt-olmos-reggiani-2017}. Let $M = G/H$ be a compact
homogeneous space, where $G = I(M)$ is the full isometry group of
$M$. Let $\mathfrak g$ be the Lie algebra of $G$, which is naturally
identified with the algebra $\mathfrak K(M)$ of Killing vector fields
on $M$. We also denote by $\mathfrak h$ the Lie algebra of the full
isotropy group $H$. Given $q \in M$, we define the \emph{Cartan
  subspace} at $q$ as   
\begin{equation*}
  \mathfrak p^q = \{X \in \mathfrak g: (\nabla X)_q = 0\},
\end{equation*}
where $\nabla$ is the Levi-Civita connection of $M$. The elements in
$\mathfrak p^q$ are called \emph{transvections} at $q$. The
\emph{symmetric isotropy algebra} at $q$ is defined by
\begin{equation*}
  \mathfrak k^q = \operatorname{span}_{\mathbb R}\{[X, Y]: X, Y \in
  \mathfrak p^q\}.
\end{equation*}

It is easy to see that $\mathfrak k^q$ is contained in $\mathfrak
h$. Let us define 
\begin{equation*}
  \mathfrak g^q = \mathfrak k^q \oplus \mathfrak p^q,
\end{equation*}
which is an involutive subalgebra of $\mathfrak g$. We denote by $G^q$
the connected Lie subgroup of $G$ with Lie algebra $\mathfrak g^q$. The
\emph{distribution of symmetry} $\mathfrak s$ of $M$ is defined by 
\begin{equation*} 
  q \mapsto \mathfrak s_q = \{X_q: X \in \mathfrak p^q\}
\end{equation*}
and it is a $G$-invariant autoparallel distribution of $M$ (that is,
integrable with totally geodesic leaves). The rank $i_{\mathfrak
  s}(M)$ of the distribution $\mathfrak s$ is known as the \emph{index
  of symmetry} of $M$, and the co-index of symmetry of $M$ is defined
as $ci_{\mathfrak s}(M) = \dim M - i_{\mathfrak s}(M)$. The integral
manifold $L(q)$ of $\mathfrak s$ by $q$ is a totally geodesic
submanifold of $M$, and moreover, it is extrinsically a globally
symmetric space. The leaves of the distribution of symmetry form a
foliation $\mathcal L$ on $M$ called the \emph{foliation of symmetry}
of $M$. Since all the leaves of the foliation of symmetry are
isometric, we will refer to $L(q)$ as the \emph{leaf of symmetry} of
$M$. Let us denote
\begin{equation*}
  \mathfrak g^{\mathfrak s} = \{X \in \mathfrak g: X \in \mathfrak
  s\}, 
\end{equation*}
which is an ideal of $\mathfrak g$ and let $G^{\mathfrak s}$ be the
corresponding normal subgroup of $G$.

\begin{remark}
  The following facts hold (see \cite{berndt-olmos-reggiani-2017}).
  \begin{enumerate}
  \item The groups $G^{\mathfrak s}$ and $G^q$ act almost effectively
    on the leaf of symmetry $L(q)$.
  \item If $\bar G^q = \{g|_{L(q)}: g \in G^q\}$ and $\bar K^q = \{h|_{L(q)}:
    h \in H\}$, then the Lie algebra of $\bar K^q$ is $\mathfrak k^q$
    (restricted to the leaf of symmetry) and $G^q/K^q$ is a symmetric
    presentation for $L(q)$. 
  \end{enumerate}
\end{remark}

The most important general result for compact homogeneous spaces
related to these topics is the following theorem. 

\begin{theorem}[\cite{berndt-olmos-reggiani-2017}]
  \label{sec:preliminaries}
  Let $M = G/H$ a compact, simply connected homogeneous Riemannian
  space with $G = I(M)$ and co-index of symmetry $k$. Assume that $M$
  does not split of a symmetric de Rham factor. Then $k \ge 2$ and
  there exists a Lie group $G'$ with the following properties.
  \begin{enumerate}
  \item $G'$ is a semisimple normal subgroup of $G$.
  \item $G'$ is transitive on $M$.
  \item $\mathfrak g = \mathfrak g^{\mathfrak s} \oplus \mathfrak g'$
    (direct sum of ideals), where $\mathfrak g'$ is the Lie algebra of
    $G'$.
  \item\label{item:1} $\dim G' \le k(k + 1)/2$.
  \item If $\dim G' = k(k + 1)/2$ then the universal cover of $G'$ is
    $Spin(k + 1)$.
  \item If $k \ge 3$ and $\dim G' = k(k + 1)/2$, the isotropy group of
    $G'$ has positive dimension.
  \end{enumerate}
\end{theorem}

In particular, the item \ref{item:1} of the above theorem gives us a
bound on the dimension of $M$ in terms of its co-index of
symmetry. Finally, recall that the Lie algebra $\bar{\mathfrak g}^q$
of $\bar G^q$ (which is isomorphic to $\mathfrak g^q$) can be
decomposed as a sum of ideals
\begin{equation*}
  \bar{\mathfrak g}^q = \hat{\mathfrak g} \oplus \bar{\mathfrak g}'^q,
\end{equation*}
where $\bar{\mathfrak g}'^q$ is the restriction of $\mathfrak g' \cap
\mathfrak g^q$ to $L(q)$ and $\hat{\mathfrak g}$ is the restriction of
$\mathfrak g^{\mathfrak s} \cap \mathfrak g^q$. (Recall that
$\mathfrak g^{\mathfrak s} \cap \mathfrak g^q$ could contain an ideal
which acts trivially on $L(q)$.)

In the case of co-index $4$, Theorem \ref{sec:preliminaries} says that
the underling manifold, up to a cover, is one of the following:
\begin{multicols}{2}
  \begin{itemize}
  \item $SO(5)/SO(2)$
  \item $SO(5)/SO(3)$
  \item $SO(5)/(SO(2) \times SO(2))$
  \item $SO(5)/(SO(3) \times SO(2))$
  \item $SO(5)/SO(4)$
  \item $SO(4)$
  \item $SO(3) \times SO(3) \times SO(3)$
  \end{itemize}
\end{multicols}

The rest of the article is devoted to decide which ones of the above
manifolds does actually admit an invariant metric with co-index of
symmetry $4$. 

\section{Inadmissible manifolds}
\label{sec:inadm-manif-2}

\begin{theorem}
  \label{sec:inadm-manif-1}
  There is not any $SO(5)$-invariant metric on $M = SO(5)/SO(2)$ with
  co-index of symmetry equal to $4$.
\end{theorem}

\begin{proof}
  Since $\dim M = 9$ and $ci_{\mathfrak s}(M) = 4$, the leaf of
  symmetry $L(q)$ is a symmetric space of dimension $5$. Since we are
  working locally, we can assume that $L(q)$ is product of a simply
  connected symmetric space of the compact type and a (possibly
  trivial) torus. So, the different possibilities for $L(q)$ are $S^5$,
  $S^4 \times T^1$, $S^3 \times S^2$, $S^3 \times T^2$, $S^2 \times
  S^2 \times T^1$, $S^2 \times T^3$ and $T^5$. 

  Let us first look at the case $L(q) = S^5$. Here $\so(6) =
  \bar{\mathfrak g}^q = \hat{\mathfrak g} \oplus \bar{\mathfrak g}'^q$
  is a simple Lie algebra, and hence one of these two ideals must be
  trivial. Since $\so(6)$ is the full isometry Lie  algebra of $L(q)$,
  and from Theorem \ref{sec:preliminaries}, $G'$ has isotropy group of
  positive dimension, we conclude that $\hat{\mathfrak g} = 0$. This
  implies that $\so(6) = \bar{\mathfrak g}^q$, which acts effectively
  on $L(q)$, must be contained in $\mathfrak g' = \so(5)$. A
  contradiction. For the case $L(q) = S^4 \times T^1$ we argue
  similarly. Here $\bar{\mathfrak g}^q = \so(5) \oplus \mathbb R$, so
  in the decomposition $\bar{\mathfrak g}^q = \hat{\mathfrak g} \oplus
  \bar{\mathfrak g}'^q$ we must have $\hat{\mathfrak g} = \mathbb R$
  and $\bar{\mathfrak g}'^q = \mathfrak g' = \so(5)$, and hence
  $\mathfrak g'$ could not be transitive on $M$, which is absurd.
  
  Assume now that $L(q) = S^3 \times S^2$, and hence $\bar{\mathfrak
    g}^q = \so(4) \oplus \so(3)$ as a direct sum of ideals, where the
  first summand corresponds to the full isometry Lie algebra of $S^3$
  and the second one is the isometry algebra of $S^2$. Since $G'$ is
  transitive on $M$, $\bar{\mathfrak g}'^q$ splits as the direct sum
  of two ideals $\mathfrak g^3 \oplus \so(3)$, where $\mathfrak g^3$
  is the Lie algebra of a transitive isometry subgroup of $S^3$ and
  the second summand is the Lie algebra of the full isometry group of
  $S^2$. We claim that $\hat{\mathfrak g} = 0$ in the decomposition
  $\bar{\mathfrak g}^q = \hat{\mathfrak g} \oplus \bar{\mathfrak
    g}'^q$. Otherwise, we must have that $\hat{\mathfrak g} \simeq
  \so(3)$ and, up to an isometry of $M$, $\bar{\mathfrak g}^q$
  decomposes in the following manner. If we identify the sphere $S^3$
  with the unit quaternions, then we can present $\bar{\mathfrak g}^q
  = \so(3)^\ell \oplus \so(3)^r$ as a direct sum of ideals isomorphic
  to $\so(3)$, where $\so(3)^\ell$ and $\so(3)^r$ are the Lie algebras
  of the left and right multiplications respectively on $S^3$. Without
  lose of generality we can assume that $\hat{\mathfrak g} =
  \so(3)^\ell$ and $\bar{\mathfrak g}'^q = \so(3)^r \oplus
  \so(3)$. Let us denote by $SO(4) = SO(3)^\ell \times SO(3)^r$
  (almost direct product) the isometry group of the factor $S^3$ of
  $L(q)$. Since $\hat{\mathfrak g} \subset {\mathfrak g}^{\mathfrak
    s}$, we have that $SO(3)^\ell$ leaves invariant the factor $S^3$
  of any other leaf of symmetry. This implies that $SO(3)^r$ does so,
  and hence $\so(3)^r$ must be contained in $\hat{\mathfrak g}$, a
  contradiction from assuming $\hat{\mathfrak g} \neq \{0\}$. So,
  $\bar{\mathfrak g}^q = {\mathfrak g}'^q = \so(4) \oplus \so(3)$ is
  the direct sum of the Lie algebras of transvections of $S^3$ and
  $S^2$. This says that de dimension of the isotropy group of $G'$ is
  greater or equal than $2$, which is a contradiction. This excludes
  the case $L(q) = S^3 \times S^2$.

  The cases $S^3 \times T^2$, $S^2 \times S^2 \times T^1$, $S^2 \times
  T^3$ and $T^5$ can be disregarded all at once with the following
  argument. In such cases the leaf of symmetry is a symmetric space
  of rank at least $3$, and $G' = SO(5)$ must contain a subgroup which
  is transitive on $L(q)$, but this is impossible since $SO(5)$ has
  rank $2$.
\end{proof}

\begin{remark}
  Recall that the proof of Theorem \ref{sec:inadm-manif-1} is
  independent of the choice of the inclusion $SO(2) \hookrightarrow
  SO(5)$, for which there are infinitely many geometric
  possibilities. 
\end{remark}

\begin{proposition}
  \label{sec:inadm-manif}
There is not any $SO(5)$-invariant metric on $M = SO(5)/(SO(3) \times
SO(2))$ with co-index of symmetry equal $4$.  
\end{proposition}

\begin{proof}
  Since $\dim M = 6$, if the metric has co-index $4$, then the leaf of
  symmetry $L(q)$ must be locally isometric to the sphere $S^2$ or the
  torus $T^2$. This implies that $\dim \bar{\mathfrak g}^q \le 3$ and
  $\hat{\mathfrak g} = \{0\}$ in the decomposition $\bar{\mathfrak
    g}^q = \hat{\mathfrak g} \oplus \bar{\mathfrak g}'^q$. On the
  other hand, we have that the isotropy group $SO(3) \times SO(2)$ of
  $G' = SO(5)$ leaves invariant $L(q)$ and hence, $\so(3) \oplus
  \so(2) \subset \bar{\mathfrak g}'^q$. This is impossible, since the
  action of $\bar G^q$ on $L(q)$ is almost effective.
\end{proof}

\begin{remark}
  As a matter of fact, the case of $M = SO(5)/SO(4)$, which is
  diffeomorphic to the sphere $S^4$, is not even under consideration
  because co-index $4$ means that $i_{\mathfrak s}(M) = 0$, and we are
  only interested in the cases where the distribution of symmetry is
  non-trivial. Nevertheless, this situation is also impossible, since
  is a well-known fact that the only $SO(5)$-invariant metric on $S^5$
  is the round one (up to scaling). This follows, for instance, from
  the fact that $SO(5)/SO(4)$ is an isotropy irreducible space (see
  \cite{wolf-1968}). 
\end{remark}

\section{Examples of spaces with co-index of symmetry $4$}

In this section we present an example of a metric with co-index of
symmetry $4$ for each of the manifolds which were not excluded in
Section \ref{sec:inadm-manif-2}.

\subsection{Double symmetric pairs}
\label{sec:double-symm-pairs}

For the first two examples we use a construction given in
\cite{olmos-reggiani-tamaru-2014} using double symmetric pairs. Let us
review briefly this argument. Let us as consider a triple
$G \supset G' \supset K'$ where $G$ is a compact Lie group and $G'$,
$K'$ are compact subgroups of $G$. Assume that $G'$ is simple and
$(G', K')$ is a symmetric pair (which cannot be of the group
type). Let $(-,-)$ be an $\Ad(G)$-invariant inner product on the Lie
algebra $\mathfrak g'$ of $G'$. Denote by
$\mathfrak g' = \mathfrak k' \oplus \mathfrak p'$ the Cartan
decomposition of $(G', K')$, where $\mathfrak k'$ is the Lie algebra
of $K'$. Recall that, since $G'$ is simple, the restriction of
$(-, -)$ to $\mathfrak g'$ is a multiple of the Killing form of
$\mathfrak g'$, and so $\mathfrak k'$ is orthogonal to $\mathfrak p'$
with respect to $(-, -)$. Let $\mathfrak m$ be orthogonal complement
of $\mathfrak k'$ with respect to $(-, -)$. Since
$\mathfrak p' \subset \mathfrak m$, we have an orthogonal
decomposition $\mathfrak m = \mathfrak m' \oplus \mathfrak p'$, where
$\mathfrak m' = (\mathfrak p')^\bot \cap \mathfrak m$. Now, we define
an inner product $\langle-, -\rangle$ on $\mathfrak m$ by asking:
\begin{align*}
  \langle\mathfrak m', \mathfrak p'\rangle = 0,
  &&
     \langle-, -\rangle|_{\mathfrak m' \times \mathfrak m'} = (-,
     -)|_{\mathfrak m' \times \mathfrak m'},  
  &&
     \langle-, -\rangle|_{\mathfrak p' \times \mathfrak p'} =
     2(-,-)|_{\mathfrak p' \times \mathfrak p'}. 
\end{align*}
Endow $M = G/K'$ with the $G$-invariant Riemannian metric induced by
the inner product $\langle-, -\rangle$ on $\mathfrak m = T_{eH}M$. We
denote such metric with same symbol $\langle-, -\rangle$. It follows
from the results in \cite{olmos-reggiani-tamaru-2014} that the
$G$-invariant distribution induced by $\mathfrak p'$ is contained in
the distribution of symmetry of $M$. Moreover, if $G/G'$ is an
irreducible symmetric space (with the normal homogeneous metric) and
$G/K'$ is not a locally symmetric space, then the distribution of
symmetry of $M$ is exactly the distribution induced by $\mathfrak p'$
and the leaf of symmetry is isometric to $G'/K'$.

\subsection{The case of $SO(5)/SO(3)$}
\label{sec:case-so5so3}

This case case was already treated in
\cite{olmos-reggiani-tamaru-2014}. Consider the standard inclusions
$SO(5) \supset SO(4) \supset SO(3)$. The construction given in
Subsection \ref{sec:double-symm-pairs} does not apply directly, since
$SO(4)$ is not simple, but this difficulty can be avoided by noticing
that the restriction of the Killing form of $\so(5)$ to $\so(4)$ is a
multiple of the Killing form of $\so(4)$. The above construction gives
an $SO(5)$-invariant metric on $SO(5)/SO(3)$ with leaf of symmetry
isometric to the sphere $S^3 = SO(4)/SO(3)$. (One should check that
$SO(5)/SO(3)$ is not a symmetric space.) Recall that, after a
rescaling of the metric, $SO(5)/SO(3)$ is isometric to the unit
tangent bundle of the $4$-sphere of curvature $2$ (with the Sasaki
metric).

Recall that in principle there are infinitely many possible
presentations of $SO(5)/SO(3)$ as a homogeneous manifolds. However, we
shall prove here that if a such a manifold admits a $SO(5)$-invariant 
Riemannian metric with co-index of symmetry $4$, the presentation is
essentially the one given by the identification
\begin{equation}
  \label{eq:4}
  SO(3) \simeq
  \begin{pmatrix}
    SO(3) & & \\
    & 1 & 0 \\
    & 0 & 1
  \end{pmatrix}
  \subset SO(5).
\end{equation}

\begin{theorem}
  \label{sec:case-so5so3-1}
  Let $M = SO(5)/H$ be a Riemannian homogeneous manifold where $H$ is
  a closed subgroup of $SO(5)$ isomorphic to $SO(3)$. Assume that $M$
  has co-index of symmetry $4$ and that the universal cover of $M$ does
  not have a symmetric de Rham factor. Then $H$ is conjugated to the
  subgroup given in (\ref{eq:4}).
\end{theorem}

\begin{proof}
  Recall that for the standard homogeneous presentation $SO(5)/SO(3)$,
  the isotropy representation is given by the $\ad$-representation,
  restricted to the reductive complement associated with the normal
  homogeneous decomposition $\so(5) = \so(3) \oplus \mathfrak m$,
  where  
  \begin{equation*}
    \mathfrak m = \left\{\left(
    \begin{smallmatrix}
      0 & 0 & 0 & -a & -b \\
      0 & 0 & 0 & -c & -d \\
      0 & 0 & 0 & -e & -f \\
      a & c & e & 0 & -g \\
      b & d & f & g & 0
    \end{smallmatrix} 
    \right): a, b, c, d, e, f, g \in \mathbb R\right\}
  \end{equation*}
  We decompose $\mathfrak m = \mathfrak m_+ \oplus \mathfrak m_{-}
  \oplus \mathfrak m_0$ into irreducible subrepresentations where
  \begin{align*}
    \mathfrak m_+ = \left\{\left(
    \begin{smallmatrix}
      0 & 0 & 0 & -a & 0 \\
      0 & 0 & 0 & -c & 0 \\
      0 & 0 & 0 & -e & 0 \\
      a & c & e & 0 & 0 \\
      0 & 0 & 0 & 0 & 0
    \end{smallmatrix}
                      \right): a, b, c \in \mathbb R\right\}
        &&
           \mathfrak m_{-} = \left\{\left(
           \begin{smallmatrix}
             0 & 0 & 0 & 0 & -b \\
             0 & 0 & 0 & 0 & -d \\
             0 & 0 & 0 & 0 & -f \\
             0 & 0 & 0 & 0 & 0 \\
             b & d & f & 0 & 0
           \end{smallmatrix}
                             \right): b, d, f \in \mathbb R\right\}
  \end{align*}
  and $\mathfrak m_0$ corresponds to the fixed points of $SO(3)$. Note
  also that $H$ acts on $\mathfrak m_{\pm}$ as the adjoint
  representation of $SO(3)$. So, in order to prove the theorem, it is
  enough to show that the isotropy representation of $M = SO(5)/H$ at
  $p \in M$ is equivalent to the representation given above. Decompose 
  \begin{equation*}
    T_pM = \mathfrak s_p \oplus \mathbb W \oplus \mathbb L
  \end{equation*}
  as an orthogonal sum of invariant subspaces, where $\dim\mathbb W =
  \dim \mathfrak s_p = 3$ and $\mathbb L$ is a line of fixed vectors
  of $H$ in $T_pM$. Note that it is enough to show that the Lie group
  morphisms $\rho_1: H \to SO(\mathfrak s_p)$ and $\rho_2: H \to
  SO(\mathbb W)$, given by the restriction of the isotropy
  representation to $\mathfrak s_p$ and $\mathbb W$, are both
  isomorphisms. Let us denote $\Phi_i = \ker\rho_i$ for $i = 1,
  2$. Notice that $\Phi_1 \cap \Phi_2$ is trivial, since we are
  assuming that the action of $SO(5)$ is effective.
  
  Suppose that $\Phi_1$ is not trivial, and let $\mathcal D^{\Phi_1}$
  be the $SO(5)$-invariant distribution on $M$ given by the fixed
  points of $\Phi_ 1$. That is, $\mathcal D^{\Phi_1}_p = \mathfrak s_p
  \oplus \mathbb L$. It is known that $\mathcal D^{\Phi_1}$ is an
  autoparallel distribution (see for instance \cite[Lemma
  5.1]{olmos-reggiani-2012}). Moreover, since $\mathfrak s$ is the
  distribution of symmetry of $M$, it follows that $\mathcal
  D^{\Phi_1}$ is strongly symmetric with respect to $SO(5)$ (see
  Appendix \ref{sec:strongly-symm-autop}). Notice that the co-rank of 
  $\mathcal D^{\Phi_1}$ is $3$. Since $M$ does not have a symmetric de
  Rham factor, we conclude that there exists a transitive semisimple
  subgroup $G' \subset SO(5)$ with $\dim G' \le 6$, which is
  absurd. So $\Phi_1$ is trivial. 

  Now assume that $\Phi_2$ is not trivial and consider the
  $SO(5)$-invariant autoparallel distribution $\mathcal D^{\Phi_2}$
  induced by the fixed vectors of $\Phi_2$. Since $\mathcal
  D^{\Phi_2}_p = \mathbb W \oplus \mathbb L$, and $\mathfrak s =
  (\mathcal D^{\Phi_2})^\bot$ is also autoparallel, it follows that
  the distribution of symmetry is parallel on $M$, and hence $M$ has
  locally a symmetric de Rham factor. This is a contradiction,
  therefore $\Phi_2$ must be trivial. 
\end{proof}

\subsection{The case $SO(5)/(SO(2) \times SO(2))$}

Now consider the standard inclusions $SO(5) \supset SO(4) \supset
SO(2) \times SO(2)$. We have here the same situation as in the above
case where the Killing form of $\so(4)$ is a scalar multiple of the
restriction of the Killing form of $\so(5)$, so the construction from
double symmetric pairs applies. Recall that $SO(4)/(SO(2) \times
SO(2))$ is the Grassmannian $G_2^+(\mathbb R^4)$ of oriented
$2$-planes in $\mathbb R^4$, which is isometric to the product of
round spheres $S^2 \times S^2$. So, the metric of Subsection
\ref{sec:double-symm-pairs} gives us a $SO(5)$-invariant metric on
$SO(5)/(SO(2) \times SO(2))$, with leaf of symmetry $G_2^+(\mathbb
R^4)$, provided it is not symmetric.

\begin{lemma}
  With the $SO(5)$-invariant metric defined in the above paragraph,
  the space $M = SO(5)/(SO(2) \times SO(2))$ is not a locally
  symmetric space.  
\end{lemma}

\begin{proof}
  Let us consider the universal covering $\tilde M = Spin(5)/(Spin(2)
  \times Spin(2))$ of $M$, where $Spin(2) \simeq SO(2)$. It is enough
  to prove that $\tilde M$ is not a globally symmetric space. Assume
  that $\tilde M$ is a symmetric space. Recall that, since $\tilde M$
  is compact and simply connected, it cannot have a flat factor. 

  Let us prove first that $\tilde M$ must be irreducible. In fact, let
  $\tilde M = M_1 \times \cdots \times M_k$ be the de Rham
  decomposition of $\tilde M$, where $M_i$ is a compact, simply
  connected, irreducible symmetric space space. Since, $Spin(5)$ is
  simple, projecting down the group $Spin(5) \subset I(\tilde M)$ to
  $I(M_i)$ we get a transitive subgroup of $I(M_i)$ isomorphic to
  $Spin(5)$ (since the kernel of this projection is a normal subgroup
  of $Spin(5)$ and $M_i$ is simply connected). In particular, since
  $\dim \tilde M = 8$, no factor $M_i$ in the decomposition of $\tilde
  M$ can be a symmetric space of the group type. Let us denote by
  $n_i$ the dimension of $M_i$. Since $10 = \dim Spin(5) \le \dim
  I(M_i) \le n_i(n_i + 1)/2$, we conclude that $k = 2$, and $n_1 = n_2
  = 4$. This implies that $\tilde M = S^4 \times S^4$ and $I_0(\tilde
  M) = Spin(5) \times Spin(5)$ (almost effective action). This is a
  contradiction, because no subgroup of $I(\tilde M)$, isomorphic to
  $Spin(5)$ can be transitive in $S^4 \times S^4$. 

  So $\tilde M$ is a simply connected, compact irreducible symmetric
  space which is not of the group type. Thus the only possibilities
  are $\tilde M = G_1^+(\mathbb H^3)$ or $\tilde M = G_2^+(\mathbb
  R^6) = SO(6)/(SO(2) \times SO(4))$. Since we note before that
  $\tilde M$ has a totally geodesic submanifold isometric to
  $G_2^+(\mathbb R^4)$, we can easily exclude the case $\tilde M =
  G_1^+(\mathbb H^3)$, which is a rank one symmetric space. The case
  $\tilde M = G_2^+(\mathbb R^6)$ is also impossible, because
  $Spin(5)$ could not act transitively on $\tilde M$.

  So, $\tilde M$ is not a symmetric space, which concludes the proof of
  the lemma.
\end{proof}

\begin{remark}
  We remark the work of Podestá \cite{podesta-2015} on constructing
  invariant metrics on generalized flag manifold, which applies to the
  homogeneous manifold $SO(5)/(SO(2) \times SO(2))$. He deals with
  Kähler and the leaves of symmetry is always an irreducible Hermitian
  symmetric space. So, our example is different from the ones given by
  Podestá, since in our case the leaf of symmetry is $G_2^+(\mathbb
  R^4) \simeq S^2 \times S^2$. In particular, the metric is not Kähler.
\end{remark}

\subsection{The case of $SO(4)$}

Let us work, for simplicity, in the universal covering group of
$SO(4)$ presented as $SU(2) \times SU(2)$. We present several families
of left invariant metrics on with co-index of symmetry $4$. First of
all, we recall the classification of homogeneous spaces with co-index
of symmetry $2$, which are all left invariant metrics on $SU(2)$ (see
\cite{berndt-olmos-reggiani-2017} or \cite{Reggiani_2018}). Let
us denote by 
\begin{align}
  \label{eq:1}
  X_1 = \frac12
  \begin{pmatrix}
    i & 0 \\
    0 & -i
  \end{pmatrix},
      &&
         X_2 = \frac12
         \begin{pmatrix}
           0 & -1 \\
           1 & 0
         \end{pmatrix},
      &&
         X_3 = \frac12
         \begin{pmatrix}
           0 & -i \\
           -i & 0
         \end{pmatrix}
\end{align}
the standard basis of $\su(2)$. Any left invariant metric on $SU(2)$,
up to isometric automorphism, is represented in the basis \eqref{eq:1}
by the symmetric definite positive matrix
\begin{equation}
  \label{eq:2}
  M(\lambda, \mu, \nu) =
  \begin{pmatrix}
    \lambda & 0 & 0 \\
    0 & \mu & 0 \\
    0 & 0 & \nu
  \end{pmatrix}, \qquad \lambda \ge \mu \ge \nu > 0,
\end{equation}
being the round metric on $SU(2)$ the one with $\lambda = \mu =
\nu$. The left invariant metrics on $SU(2)$ with co-index of symmetry
$2$ are, up to isometry and scaling, the associated with the matrices
$M(\lambda, \lambda - 1, 1)$, with $\lambda > 2$; $M(\lambda, 1, 1)$,
with $\lambda > 1$; and $M(1, 1, \nu)$, with $0 < \nu < 1$. The last
two families parameterise the so-called Berger spheres. Let us denote
by $SU(2)_{\lambda, \mu, \nu}$ the group $SU(2)$ endowed with the left
invariant metric represented by $M(\lambda, \mu, \nu)$.

From the previous comments, one can easily construct a large number of
examples of left invariant metrics on $SU(2) \times SU(2)$ with
co-index of symmetry $4$. Namely, denote by $(\lambda, \mu, \nu)$ one
of the triples $(\lambda, \lambda - 1, 1)$, $(\lambda, 1, 1)$ or $(1,
1, \nu)$ with the restrictions imposed above, and similarly assume
that $(\lambda', \mu', \nu')$ takes the form $(\lambda', \lambda' - 1,
1)$, $(\lambda', 1, 1)$ or $(1, 1, \nu')$. So, one can form six
$2$-parameter families of spaces $SU(2)_{\lambda, \mu, \nu} \times
SU(2)_{\lambda', \mu', \nu'}$ with co-index of symmetry $4$. Note that
these spaces are Riemannian products, but they do not split of a
symmetric de Rham factor and so they satisfies the hypothesis of
Theorem \ref{sec:preliminaries}. 

We present another family of examples, which appears in the
classification of naturally reductive spaces of dimension up to $6$
given by Agricola, Ferreira and Friedrich
\cite{agricola-ferreira-friedrich-2014}, and whose index of symmetry
is computed by using the results in
\cite{olmos-reggiani-tamaru-2014}. We present $SU(2) \times SU(2)$ as
the homogeneous manifold $G/H$ where $G = SU(2) \times SU(2) \times
SU(2)$ modulo the diagonal subgroup $H = \{(g, g, g): g \in
SU(2)\}$. Denote by $\mathfrak g = \su(2) \oplus \su(2) \oplus \su(2)$
and $\mathfrak h = \{(X, X, X): X \in \su(2)\}$ their respective Lie
algebras. Put 
\begin{align*}
  \mathfrak m_1 & = \{\tilde X = (X, aX, bX): X \in \su(2),\, a, b,
                  \in \mathbb R\}, \\
  \mathfrak m_2 & = \{\tilde Y = (Y, cY, dY): Y \in \su(2),\, c, d,
                  \in \mathbb R\}.
\end{align*}
If we ask $0 \neq (a - 1)(d - 1) - (b - 1)(c - 1) = \det
\begin{pmatrix}
  1 & 1 & 1 \\
  1 & a & b \\
  1 & c & d
\end{pmatrix}
$ then $\mathfrak m = \mathfrak m_1 \oplus \mathfrak m_2$ is a reductive
complement of $\mathfrak h$ and, for each $\lambda > 0$, the inner
product on $\mathfrak m$ defined by
\begin{equation*}
  \langle(\tilde X_1, \tilde Y_1), (\tilde X_2, \tilde Y_2)\rangle =
  -\frac12\left(\trace(X_1X_2) + \frac1{\lambda^2}\trace(Y_1Y_2)\right)
\end{equation*}
induces a naturally reductive metric on $G/H$. It is easy to see that
the set of fixed vectors of the isotropy representation is a
$2$-dimensional subspace of $\mathfrak m$. So in the generic case
(when the metric is not symmetric), it follows from
\cite{olmos-reggiani-tamaru-2014} that the co-index of symmetry is
equal to $4$.

\subsection{The case of $SO(3) \times SO(3) \times SO(3)$}

We can form metrics with co-index of symmetry $4$ in $SO(3) \times
SO(3) \times SO(3)$ by taking the product of the bi-invariant
(symmetric) metric on the first factor and one of the metrics
presented in the above case on the others two factors. So, we have a
rank $5$ distribution of symmetry in a $9$-dimensional homogeneous
space. Recall that this example is not exactly in the hypothesis of
Theorem \ref{sec:preliminaries} it has co-index of symmetry $4$
though. Sadly, we have not been able to find an irreducible example yet. 

\appendix
\section{Strongly symmetric autoparallel distributions}
\label{sec:strongly-symm-autop}

We introduce shortly the concept of strongly symmetric
distribution. The full reference for this appendix is the article
\cite{berndt-olmos-reggiani-2017}. Let $M = G/H$ be a compact
homogeneous Riemannian manifold. Assume that $G$ is connected and its
action on $M$ is effective. We say that a $G$-invariant autoparallel
distribution $\mathcal D$ is \emph{strongly symmetric} with respect to
$G$ if every integral manifold $L(p)$ is a globally symmetric space
and for each $v \in \mathcal D_p$ there exists a Killing field $X$ on
$M$, which is induced by $G$, such that $X_p = v$ and $X|_{L(p)}$ is
parallel at $p$.

\begin{example}
  The distribution of symmetry of $M$ is strongly symmetric with
  respect to the full isometry group.
\end{example}

\begin{example}
  [{\cite[Lemma 3.11]{berndt-olmos-reggiani-2017}}]
  If $\mathcal D$ is strongly symmetric with respect to $G$ and
  $\mathcal D'$ is a $G$-invariant autoparallel distribution such that
  $\mathcal D \subset \mathcal D'$ and $\rank\mathcal D' -
  \rank\mathcal D = 1$, then $\mathcal D'$ is strongly symmetric with
  respect to $G$.
\end{example}

Theorem \ref{sec:preliminaries} has a weaker version for strongly
symmetric distributions.

\begin{theorem}
  [{\cite[Theorem 3.7]{berndt-olmos-reggiani-2017}}]
  Let $\mathcal D$ be strongly symmetric with respect to $G$
  and let $k = \dim M - \rank \mathcal D$. Assume that $M$ does not
  have a symmetric de Rham factor with associated parallel
  distribution contained in $D$. Then there exists a transitive
  semisimple Lie group $G' \subset G$ such that $\dim G' \le k(k +
  1)/2$. Moreover, equality holds if and only the Lie algebra of $G'$
  is isomorphic to $\so(k + 1)$.
\end{theorem}

\bibliography{/home/silvio/Dropbox/math/bibtex/mybib.bib}
\bibliographystyle{amsalpha}
\end{document}